\documentclass{article}
\usepackage{graphicx}
\usepackage[dvipsnames]{xcolor}
\usepackage{amsfonts}
\usepackage{amsmath}
\usepackage{amssymb}
\usepackage{ytableau}
\usepackage{amsthm}
\usepackage{tikz}
\usepackage{tikz-cd}
\usepackage{mathtools}
\usepackage{pifont}
\usetikzlibrary{positioning,fit, shapes.geometric}
\ytableausetup{centertableaux}
\newcommand{\blankbox}{*(Gray)\;}
\newcommand{\amp}{&}

\definecolor{c1}{HTML}{D81B60}
\definecolor{c2}{HTML}{1E88E5}
\definecolor{c4}{HTML}{004D40}
\definecolor{c3}{HTML}{FFC107}

\newtheorem{theorem}[equation]{Theorem}
\newtheorem{proposition}[equation]{Proposition}

\newtheorem{corollary}[equation]{Corollary}
\newtheorem{conjecture}[equation]{Conjecture}

\newtheorem{example}[equation]{Example}
\newtheorem{remark}[equation]{Remark}
\newtheorem{definition}[equation]{Definition}
\usepackage{adjustbox}

\newtheorem{innercustomthm}{Theorem}
\newenvironment{customtheorem}[1]
  {\renewcommand\theinnercustomthm{#1}\innercustomthm}
  {\endinnercustomthm}

\newcommand{\Z}{\mathbb{Z}}
\newcommand{\BB}{\mathcal{B}}

\newcommand{\RSK}{\mathrm{RSK}}
\newcommand{\bRSK}{\mathrm{bRSK}}

\newcommand{\bitableau}[1]{\ytableausetup{boxsize=2em}\begin{ytableau}#1\end{ytableau}\ytableausetup{boxsize=normal}}
\newcommand{\bi}[2]{\resizebox{!}{1em}{$\begin{array}{c}
    #1 \\ #2
\end{array}$}}

\newcommand{\boldbi}[2]{\resizebox{!}{1em}{$\begin{array}{c}
    \mathbf{#1} \\ \mathbf{#2}
\end{array}$}}

\title{Kronecker Coefficients, Crystals, and Bitableaux}
\author{Nate Harman and Alexander N. Wilson}
\date{July 18, 2025}

\begin{document}

\maketitle

\begin{abstract}
    What might a combinatorial interpretation of the Kronecker coefficients even look like?   We introduce a class of combinatorial objects called bitableaux, which we believe are a natural candidate, and we formulate a purely combinatorial problem which if resolved would give a combinatorial interpretation of the Kronecker coefficients.  We make some partial progress on this problem -- enough to extract a combinatorial expansion for a Kronecker product of Schur functions in the monomial basis. We also explain how in this framework finding a combinatorial interpretation for Kronecker coefficients can be thought of as looking for a generalization of the RSK and dual RSK insertion algorithms. 
\end{abstract}

\section{Introduction}

The Kronecker coefficients are notoriously mysterious structure constants in the representation theory of general linear and symmetric groups.  A long-standing open problem in combinatorial representation theory is to provide a positive combinatorial interpretation for these coefficients. 
Outside of a few special cases, little progress has been made toward finding a combinatorial interpretation.   

In this paper we introduce combinatorial objects called (lexicographic) bitableaux designed to study Kronecker coefficients. Here is an example:
\begin{align*}
        \bitableau{
                \bi11 & \bi12 & \bi12 &  \bi21\\
                \bi22 & \bi22\\
                \bi31
            }
    \end{align*}
We will define lexicographic bitableaux and various quantities associated to them in section \ref{BitabIntro}. For now, though, we will just state the motivating conjecture:

\begin{conjecture}\label{mainconj}
    The Kronecker coefficients are counted by lexicographic bitableaux of given shape and weights that satisfy a pair of Yamanouchi reading word conditions.
\end{conjecture}
 
 One important caveat about this conjecture is that we only have a good candidate procedure to extract one of the reading words. How to extract the second reading word remains the key mystery left to resolve in order to carry out the program outlined in this paper. 
 
 The primary goal of this paper is to motivate and provide evidence for this conjecture.  Our main new technical result is partial progress towards the conjecture, and uses this framework to give a positive combinatorial interpretation for the \emph{monomial} expansion of a Kronecker product of Schur functions:

 \begin{theorem}\label{thm:monominterp}
     The Kronecker product of two Schur functions can be expanded in the monomial basis as:
    \begin{align*}
        s_\lambda\ast s_\nu&=\sum_{\mu}d(\lambda,\mu,\nu)m_\mu.
    \end{align*}
    where $d(\lambda,\mu,\nu)$ counts the number of bitableaux $T$ of shape $\lambda$, with weights $a(T)= \mu$ and $b(T)= \nu$, and with Yamanouchi sort-by-top reading word $w(T)$. 
 \end{theorem}
 
 The remaining sections of the paper are organized as follows.

\begin{itemize}
    \item   In Section \ref{prelims} we will recall some necessary background material about representations of general linear groups, symmetric functions, Kronecker coefficients, RSK insertion, and crystals.  
    
    \item In Section \ref{BitabIntro} we introduce lexicographic bitableaux, and define various quantities associated to them.  We then formulate the crystal basis approach to the Kronecker coefficients problem, which is the motivation behind Conjecture \ref{mainconj}.

    \item Section \ref{sec:glm-crystal} is the main technical section of the paper.  We construct a $\mathfrak{gl}_m$-crystal on the set of bitableaux with entries in $[n] \times [m]$, such that the crystal operators preserve the $\mathfrak{gl}_n$-weights.  Using this crystal structure, we then deduce Theorem \ref{thm:monominterp}. 

    \item In Section \ref{sec:known} we look at some known cases and put them into our framework.  We explain how the RSK and dual RSK insertion algorithms can be thought of as a special case of our program -- or alternatively, how our program can be thought of as searching for a generalization of RSK.  We also look at the Kronecker tableaux of Ballentine and Orellana \cite{ballantine2005combinatorial}, and observe that they seem to fit neatly into our framework as well. 

    \item Finally, in Section \ref{sec:examples} we close out by giving some examples of crystal structures and partial crystal structures on sets of bitableaux.  
    
\end{itemize}

As of writing this, we have been unable to carry out this program in full to find a combinatorial interpretation of the Kronecker coefficients.  Nevertheless, having tried a number of ideas that ``almost" work, or work in special cases, we believe that something along these lines is possible.   So we hope you treat this paper as an open invitation:  If you have any ideas on how to extract these reading words or construct these crystals, we'd love to hear them. 

\subsection*{Acknowledgments}
We'd like to thank the organizers and our fellow participants in MAA Project NExT 2023, where this collaboration started. The first author was supported by NSF grant DMS-2401515.

\section{Preliminaries} \label{prelims}

\subsection{Polynomial Representations and Schur Functions}

A representation of $GL_n(\mathbb{C})$ is a \emph{polynomial representation} if it is isomorphic to a summand of a direct sum of tensor powers $(\mathbb{C}^n)^{\otimes k}$ of the defining representation.  Equivalently this means that the representation map $\phi: GL_n(\mathbb{C}) \to GL_N(\mathbb{C})$ is given by polynomial functions in the matrix entries. 

The irreducible polynomial representations of $GL_n(\mathbb{C})$ are indexed by partitions $\lambda$ with at most $n$ parts.  They are given by applying an appropriate Schur functor to the defining representation $\mathbb{C}^n$.  The \emph{character} of a representation is given by taking the trace of the representation evaluated at a diagonal matrix:
\[
\begin{bmatrix}
x_1 & 0 &  \dots & 0\\
0 & x_2  & \dots & 0 \\
\vdots & \vdots & \ddots & \vdots \\
0 & 0 & \dots & x_n
\end{bmatrix}
\]
If we take the character of a polynomial representation, this will be a symmetric polynomial in the entries $x_i$. The \emph{Schur polynomial} $s_\lambda(x_1, x_2, \dots, x_n)$ is the character of the irreducible polynomial representation $S^\lambda(\mathbb{C}^n)$.

Combinatorially the Schur polynomial can be described by:

$$s_\lambda(\mathbf{x}) =  s_\lambda(x_1, x_2, \dots, x_n) = {\hspace{-.3cm}} \sum_{ T \in SSYT(\lambda, n)} {\hspace{-.3cm}}x_1^{T(1)}x_2^{T(2)}\dots x_n^{T(n)} = {\hspace{-.3cm}} \sum_{ T \in SSYT(\lambda, n)} {\hspace{-.3cm}} \mathbf{x}^{c(T)}$$ where $SSYT(\lambda,n)$ denotes the set of semistandard Young tableaux of shape $\lambda$ and entries in $\{1, 2, \dots, n \}$, and $T(k)$ denotes the number of $k$'s in the Young tableau $T$.

Taking characters defines an isomorphism between the Grothendieck ring of polynomial representations $K_0(Rep^{pol}(GL_n(\mathbb{C}))$ and the ring of symmetric polynomials in $n$ variables $\Lambda_n$.  As $n$ tends to infinity, this stabilizes to an isomorphism between the Grothendieck ring of polynomial functors $K_0(\mathcal{P}ol)$ and the ring of symmetric functions $\Lambda$.

These Grothendieck rings come equipped with a natural bilinear form defined by setting
$$\langle [V], [W] \rangle \coloneq \dim(Hom(V,W))$$
for representations $V$ and $W$, and extending linearly to virtual representations. Transferring this to $\Lambda_n$ or $\Lambda$ defines the Hall inner product, which can alternatively be characterized as the unique inner product for which the Schur functions form an orthonormal basis.

\subsection{Kronecker Coefficients}

The most well-known definition of the \emph{Kronecker coefficients} $g(\lambda,\mu,\nu)$ are as the multiplicities in the decomposition of a tensor product of irreducible symmetric group representations:
$$ S^\lambda \otimes S^\mu \cong \bigoplus_\nu g(\lambda,\mu,\nu) S^\nu $$
For our purposes though it will be more convenient to pass through Schur-Weyl duality, and think of them as the branching multiplicities when restricting from $GL(V \otimes W)$ to $GL(V) \times GL(W)$:
$$ S^\nu(V \otimes W) \cong \bigoplus_{\lambda, \mu} g(\lambda,\mu,\nu) S^\lambda(V) \otimes S^\mu(W) $$

This motivates the so-called \emph{Kronecker comultiplication} of symmetric functions $\Delta: \Lambda \to \Lambda \otimes \Lambda$ that sends a symmetric function $p(\mathbf{z})$ to $p(\mathbf{xy})$ (where $p(\textbf{z})$ is shorthand for $p(z_1, z_2, \dots, z_n)$ and $p[\textbf{xy}]$ denotes the variable substitution $p(x_1y_1, x_1y_2 \dots, x_ny_m)$). Taking characters of the above decomposition gives the following expression for the Kronecker coproduct in the Schur basis: 

 $$s_\lambda[\textbf{xy}] = \sum_{\mu, \nu}g(\lambda,\mu, \nu)s_\mu(\textbf{x})s_\nu(\textbf{y})$$

The \emph{Kronecker product} of symmetric functions is defined as the map dual to Kronecker comultiplication with respect to the Hall inner product:
$$f * g \coloneq \langle f , \Delta(g) \rangle$$
In particular, in the Schur basis this means:

$$s_\lambda * s_\mu = \sum_\nu g(\lambda,\mu,\nu) s_\nu$$
While Theorem \ref{thm:monominterp} is stated in the language of the Kronecker product, we will primarily work with the Kronecker coproduct.

\subsection{RSK Insertion}

Given a semistandard Young tableau $T$ and a value $i$, one can insert $i$ into a row $R$ of $T$ as follows. If $i$ is at least as large as every entry in $R$, place $i$ at the end of $R$. Otherwise, $i$ replaces the left-most entry $j$ strictly larger than $i$, and $j$ is inserted into the next row of $T$. Given a word $w=w_1w_2\cdots w_\ell$, write $P(w)$ for the result of beginning with an empty tableau and successively inserting each value $w_1,w_2,\ldots,w_\ell$ into the first row.

\begin{example}
    \[P(211323)=\begin{ytableau}
        1 & 1 & 2 & 3 \\
        2 & 3
    \end{ytableau}\]
\end{example}

The \emph{RSK correspondence} takes a lexicographic biword \[\mathbf{w}=\left(\begin{array}{cccc}
     a_1 & a_2 & \cdots & a_\ell  \\
     b_1 & b_2 & \cdots & b_\ell
\end{array}\right)\] to a pair \[\RSK(\mathbf{w})=(P,Q)\] of semistandard tableaux where $P$ is obtained by inserting the word $b_1b_2\cdots b_\ell$ and $Q$ is obtained by recording $a_i$ in the same location as the box created upon inserting $b_i$. A notable symmetry property of RSK is that if $\mathbf{w}'$ is the biword obtained by exchanging the rows of $\mathbf{w}$ and sorting it lexicographically, then \[\RSK(\mathbf{w}')=(Q,P).\]

Let $W$ denote the monoid of words made up of positive integers with the operation of concatenation. The \emph{plactic monoid} is the quotient of $W$ by the relation setting $w\equiv u$ if $P(w)=P(u)$. This notion of equivalence was introduced by Knuth~\cite{knu:pmg}, but the monoid perspective was first considered by Lascoux and Sch\"utzenberger~\cite{LS:mp}.

Effectively, the elements of the plactic monoid correspond to semistandard Young tableaux, and the operation can be described entirely in terms of these tableaux via an operation called jeu de taquin. The jeu de taquin operation is made up of a sequence of slides of the form \begin{align*}
    \begin{ytableau}
        \none \amp a\\
        b
    \end{ytableau} &\rightarrow \begin{ytableau}
        b \amp a
    \end{ytableau} \text{ if } b\leq a \\
    \intertext{or}
    \begin{ytableau}
        \none \amp a\\
        b
    \end{ytableau} &\rightarrow \begin{ytableau}
        a \\ b
    \end{ytableau} \text{ if } b>a
\end{align*} where we imagine that every tableau has cells of infinite labels to its south and east. See Example~\ref{ex:jdt} for an example of the process and see \cite[Section 1.2]{fulton1997young} for more details. 

\begin{example}\label{ex:jdt} The product of the tableaux \[\begin{array}{ccc}
    \begin{ytableau}
     1 & 3\\
     2
\end{ytableau} & \text{ and } & \begin{ytableau}
     1 \amp 1 \amp 2\\
     2 \amp 3
\end{ytableau}
\end{array}\] in the plactic monoid is given by placing the first tableau southwest of the second to form the following skew tableau. We perform the first sequence of slides available to us:
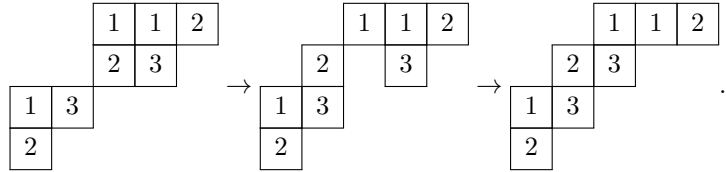
\[\begin{ytableau}
    \none & \none & 1 & 1 & 2\\
    \none & \none & 2 & 3\\
    1 & 3\\
    2
\end{ytableau}\rightarrow \begin{ytableau}
    \none & \none & 1 & 1 & 2\\
    \none & 2 & \none & 3\\
    1 & 3\\
    2
\end{ytableau}\rightarrow \begin{ytableau}
    \none & \none & 1 & 1 & 2\\
    \none & 2 & 3\\
    1 & 3\\
    2
\end{ytableau}.\] Continuing to apply these slides, we eventually obtain the straight-shape tableau 
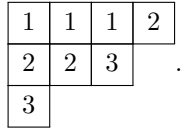
\[\begin{ytableau}
    1 & 1 & 1 & 2\\
    2 & 2 & 3\\
    3
\end{ytableau}.\]
Moreover this final tableau does not depend on the choices of which slide to do at a given step. 
\end{example}

\subsection{Crystals and Semistandard Young Tableaux}

Crystals are combinatorial objects that encode the structure of representations of reductive Lie algebras. For our purposes, we will only be interested in the Lie algebras $\mathfrak{gl}_n$ and $\mathfrak{gl}_n \times \mathfrak{gl}_m$, and all crystals will be assumed to be finite. 
A $\mathfrak{gl}_n$-crystal consists of the following data:

\begin{enumerate}
    \item An underlying finite set $\mathcal{B}$. 
    \item A weight function $\omega: \mathcal{B} \to \mathbb{Z}^n$.
    \item $n-1$ raising operators $e_i: \mathcal{B} \to \mathcal{B} \cup \{ 0 \}$  for $i = 1, \dots , n-1$.
    \item $n-1$ lowering operators $f_i: \mathcal{B} \to \mathcal{B} \cup \{ 0 \}$ for $i = 1, \dots , n-1$.
\end{enumerate}
satisfying certain axioms.  A $\mathfrak{gl}_n \times \mathfrak{gl}_m$-crystal is a single underlying set equipped with both a $\mathfrak{gl}_n$-crystal and a $\mathfrak{gl}_m$-crystal structure, such that the raising and lowering operators from the two crystal structures commute with one another.  

The crystals we construct will be obtained by constructing a weight preserving bijection between the sets we are interested in and certain known crystals, and pulling back the crystal structure. As such the axioms will hold automatically, so we will not describe them here. We refer to \cite{bump2017crystal} for the general theory.  

An element $b \in \mathcal{B}$ is \emph{highest weight} if $e_i(b) = 0$ for each $i$.  We say that a crystal is irreducible if it cannot be decomposed into a disjoint union of two smaller crystals.  The following proposition summarizes the close connection between crystals and representations. 

\newpage

\begin{proposition} \textbf{(See \cite{bump2017crystal}, Chapter 13)}
    \begin{enumerate}
        \item For each dominant weight $\lambda$ there is an irreducible $\mathfrak{g}$-crystal $B_\lambda$ that has a unique highest weight element of weight $\lambda$. Similarly, there is a unique irreducible $\mathfrak{g}$-representation $V(\lambda)$ with a highest weight vector of weight $\lambda$. 

        \item Every irreducible $\mathfrak{g}$-crystal is isomorphic to $B_\lambda$ for some $\lambda$, and every irreducible finite dimensional representation of $\mathfrak{g}$ is isomorphic to $V(\lambda)$ for some $\lambda$. 

        \item The number of elements of $B_\lambda$ of weight $\mu$ is equal to the dimension of the $\mu$ weight space in $V(\lambda)$.

    \end{enumerate}
\end{proposition} 

The basic example of a $\mathfrak{gl}_n$-crystal is on the set $\mathcal{B} = SSYT(\lambda,n)$ of semistandard Young tableaux of shape $\lambda$ and entries in $[n]$.   The weight function just records how many of each number occur in the tableaux. For example,

$$\omega \left(\;
\begin{ytableau}
            1 & 1 & 1 & 2 & 2 & 3 \\
            2 & 2 & 3 \\
            3 & 4 
        \end{ytableau} \;\right)
        = (3,4,3,1)$$
because there are $3$ ones, $4$ twos, $3$ threes, and $1$ four. Next we will explain the procedure for computing the raising and lowering operators for this crystal.  

First we extract the \emph{reading word} of the tableau, by reading its entries left-to-right in each row starting with the bottom row. 
$$
\begin{ytableau}
            1 & 1 & 1 & 2 & 2 & 3 \\
            2 & 2 & 3 \\
            3 & 4 
        \end{ytableau}
        \longrightarrow
        34223111223
$$
  The crystal raising operator $e_i$ will turn an $i+1$ into an $i$ (or send the tableaux to zero), and likewise the lowering operator $f_i$ will turn an $i$ into an $i+1$.\footnote{Yes,  the raising operator decreases an entry and the lowering operator increases an entry.  The ``raising" and ``lowering" is with respect to the dominance order on $\mathfrak{gl}_n$-weights.} 

In order to determine which $i$ or $i+1$ is affected, we apply the \emph{parenthesization procedure} to the reading word:  We convert every $i$ to a ``$)$" and every $i+1$ to a ``$($" ignoring all other entries.  Here is an example when $i = 2$:

$$ 34223111223$$
$$( \ \  ) )( \ \ \ \ \ \ ))( $$ 

We then change either the rightmost unmatched ``$)$"  to a ``$($" if we are applying a lowering operator, or the leftmost unmatched ``$($"  to a ``$)$" if we are applying a raising operator. Applying the lowering operator $f_2$ to our running example and converting back to a tableau we get:

\newpage

$$( \ \  ) )( \ \ \ \ \ \ )(( $$
$$ 34223111233$$
$$
\begin{ytableau}
            1 & 1 & 1 & 2 & 3 & 3 \\
            2 & 2 & 3 \\
            3 & 4 
        \end{ytableau}
$$

so $$f_2 \left(\;
\begin{ytableau}
            1 & 1 & 1 & 2 & 2 & 3 \\
            2 & 2 & 3 \\
            3 & 4 
        \end{ytableau} \;\right)
        =
\begin{ytableau}
            1 & 1 & 1 & 2 & 3 & 3 \\
            2 & 2 & 3 \\
            3 & 4 
        \end{ytableau}
$$
If there are no unmatched ``("s the raising operator sends the tableaux to zero.  A reading word is a \emph{Yamanouchi} word if this happens for every pair $(i , i+1)$, or in other words, if for each $i$ every suffix of the word has at least as many $i$'s as $i+1$'s. Tableaux with Yamanouchi reading word correspond to highest weight vectors in the corresponding representations. 

Note that while our running example was for a tableau of partition shape $\lambda$, the same procedure works for skew shapes $\lambda/\mu$ as well.  Moreover, if we ignore the steps in which we convert back and forth between tableaux and reading words, this parenthesization procedure also defines a crystal structure on the set $Word(n,k)$ of words of a fixed length $k$ and alphabet $[n]$.

\subsubsection{Tensor products}

A crystal tensor product rule is a process for taking two $\mathfrak{g}$-crystals on sets $B$ and $B'$ and constructing a $\mathfrak{g}$-crystal on $B \times B'$ that encodes the tensor product of the corresponding representations. 

 Note that we were careful to say \emph{a} tensor product rule rather than \emph{the} tensor product rule as there are in fact a number of different ways to do this.  For example if we compute $X \otimes Y$ via the standard tensor product rule we get something different than if we first swap the factors, compute $Y \otimes X$ the same way, and then swap back.  Here are two tensor product rules for $\mathfrak{sl_2}$ applied to $B(2) \otimes B(4)$:

\medskip

\begin{center}
\begin{tikzpicture}[
    vertex/.style={circle, fill=black, inner sep=2pt}  
  ]
  \foreach \row in {1,2,3}{
    \foreach \col in {1,...,5}{
      \node[vertex] (G1-\row-\col) at ({(\col-1)*1cm}, {-(\row-1)*1cm}) {};
    }
  }
  \foreach \c in {1,...,4}
    \draw[thick,->] (G1-1-\c) -- (G1-1-\the\numexpr\c+1\relax);
  \foreach \c in {1,...,3}
    \draw[thick,->] (G1-2-\c) -- (G1-2-\the\numexpr\c+1\relax);
  \foreach \c in {1,2}
    \draw[thick,->] (G1-3-\c) -- (G1-3-\the\numexpr\c+1\relax);
  \draw[thick,->] 
    (G1-1-5) -- (G1-2-5);
   \draw[thick,->] 
    (G1-2-5) -- (G1-3-5);
  \draw[thick,->] (G1-2-4) -- (G1-3-4);

  \begin{scope}[xshift=6cm]                     
    \begin{scope}[shift={(2cm,-1cm)},            
                   rotate=180,                   
                   shift={(-2cm,1cm)}]          
      \foreach \row in {1,2,3}{
        \foreach \col in {1,...,5}{
          \node[vertex] (G2-\row-\col) at ({(\col-1)*1cm}, {-(\row-1)*1cm}) {};
        }
      }
      \foreach \c in {1,...,4}
         \draw[thick,<-]  (G2-1-\c) -- (G2-1-\the\numexpr\c+1\relax);
      \foreach \c in {1,...,3}
         \draw[thick,<-]  (G2-2-\c) -- (G2-2-\the\numexpr\c+1\relax);
      \foreach \c in {1,2}
         \draw[thick,<-]  (G2-3-\c) -- (G2-3-\the\numexpr\c+1\relax);
       \draw[thick,<-] 
        (G2-1-5) -- (G2-2-5);
        \draw[thick,<-]
        (G2-2-5) -- (G2-3-5);
      \draw[thick,<-] (G2-2-4) -- (G2-3-4);
    \end{scope}
  \end{scope}
\end{tikzpicture}

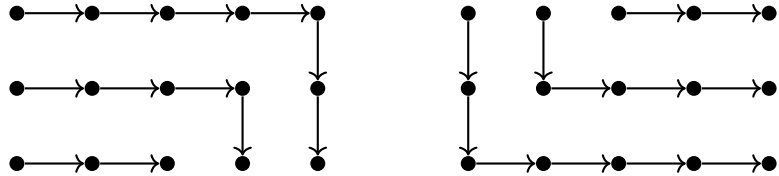
\small{Fig: The standard tensor product rule (left) and a twisted tensor product rule (right). They are different (but isomorphic) crystal structures on the same set.}
\end{center}

 We will see in some examples that sometimes it is convenient to switch between these different tensor product rules. We will not state the standard rule in full generality here (see \cite{bump2017crystal} Section 2.3 for details), but we will explain how it works for these crystals on words and tableaux.  

We will start with how the standard tensor product rule looks for the crystals on words. There is an isomorphism of crystals
$$Word(n,k) \otimes Word(n, \ell) \cong Word(n, k+\ell)$$
sending a pair of words $(w_1, w_2)$ to their concatenation $w_1w_2$, and the standard tensor product of crystals is defined such that this is an isomorphism.  We note that instead sending $(w_1, w_2)$ to $w_2w_1$ defines a different tensor product rule.

For the crystals $B(\lambda)$ and $B(\mu)$ on semistandard Young tableaux their tensor product can be identified with a union of irreducible crystals $B(\nu)$ of partitions $\nu$.  Given a pair of tableaux a new tableau is computed by extracting their reading words,  concatenating them, and converting the new reading word back to a tableaux using RSK insertion. Equivalently one can take the two tableaux, combine them into a skew shape, and then perform jeu de taquin to convert it to a partition shape as in Example~\ref{ex:jdt}.

\section{Lexicographic Bitableaux} \label{BitabIntro}

Now we introduce the main combinatorial objects of study: lexicographic bitableaux.  A lexicographic bitableau of shape $\lambda$ is a filling of the boxes of $\lambda$ with ordered pairs $(a,b)$ with $a \in [n]$ and $b\in [m]$ with the property that the entries are weakly increasing lexicographically along rows, and strictly increasing lexicographically along columns---that is, it is semistandard with respect to the lexicographic ordering on pairs. Write $\BB_\lambda$ for the set of all lexicographic bitableaux of shape $\lambda$ with entries in $\Z_{>0}\times\Z_{>0}$. Write $\BB_\lambda(n,m)$ for the subset consisting of lexicographic bitableaux of shape $\lambda$ and entries in $[n]\times[m]$.

There are other ways one could put a total order on the set of ordered pairs $(a,b)$ with $a \in [n]$ and $b \in [m]$, which would lead to other notions of bitableaux (other than lexicographic). We will not consider any of those here, so for the rest of the paper the word ``bitableaux" will refer only to lexicographic bitableaux\footnote{The term ``bitableau'' also appears in a different sense in the literature concering representations of the hyperoctahedral groups to refer to a pair of tableaux (see e.g. \cite{adin2017character}).}.  

\medskip

\begin{example}
    Of the tableaux

    \begin{align*}
        \begin{array}{ccc}
             T_1=\bitableau{
                \bi12 & \bi21\\
                \bi22 & \bi22\\
                \bi31
            }, & T_2=\bitableau{
                \bi12 & \bi21\\
                \bi32 & \bi22\\
                \bi33
            },\text{ and} & T_3=\bitableau{
                \bi12 & \bi22\\
                \bi22 & \bi22\\
                \bi31
            },
        \end{array}
    \end{align*} only $T_1$ is an example of a lexicographic bitableau. Note that $T_2$ has a decrease in its second row and $T_3$ has a repeat within its second column.
\end{example}

\medskip

Associated to a bitableau we have two weights:

\begin{enumerate}
    \item The $a$-weight is an $n$-tuple recording how many times each number appears as the first coordinate of an entry in the bitableau. We denote this tuple by $a(T)$.

    \item The $b$-weight is an $m$-tuple recording how many times each number appears as the second coordinate of an entry in the bitableau. We denote this tuple by $b(T)$.
\end{enumerate}

\medskip

\begin{example}\label{ex:bitab_weights}
    For the lexicographic bitableau \begin{align*}
        T=\bitableau{
                \bi12 & \bi21\\
                \bi22 & \bi22\\
                \bi31
            },
    \end{align*} its weights are \begin{align*}
        a(T)&=(1,3,1) \text{ and}\\
        b(T)&=(2,3,0).
    \end{align*}
\end{example}

\medskip

The reason we are looking at these bitableaux is that they give us a natural combinatorial description of the monomial expansion for the so-called Kronecker co-product on symmetric functions.

\begin{align}
    s_\lambda[\textbf{xy}] = \sum_{T \in \mathcal{B}_\lambda}\textbf{x}^{a(T)}\textbf{y}^{b(T)} = \sum_{\mu, \nu}g(\lambda,\mu, \nu)s_\mu(\textbf{x})s_\nu(\textbf{y}) \label{eq:weights}
\end{align}
where the coefficients $g(\lambda,\mu, \nu)$ are the Kronecker coefficients.

To see this, recall that if $\textbf{z}_I$ is a collection of variables indexed by a totally ordered set $I$, then the Schur function $s_\lambda[\textbf{z}_I]$ is a sum over all tableaux of shape $\lambda$ with entries in $I$ that are semistandard with respect to the total order.  If we take $I = [n] \times [m]$ ordered lexicographically and then substitute $z_{(i,j)} = x_iy_j$ we get Equation \ref{eq:weights} above.

We will define a couple of reading words for bitableaux. For each of these words, we will read entries across rows left to right, starting at the bottom row and moving up. What distinguishes them is how they group the boxes of $T$. For $T\in\BB_\lambda$, let \begin{itemize}
    \item $w(T)$ be the word obtained by first reading the bottom entries of boxes with top entry one, then the bottom entries of boxes with top entry two, and so on; and
    \item $w'(T)$ be the word obtained by reading the bottom entries of boxes with the largest top entry, then the bottom entries of boxes with the second largest top entry, and so on.
\end{itemize}

\begin{example} The bitableau in Example \ref{ex:bitab_weights} has the following reading words.
    \begin{align*}
        w(T)&=22211\\
        w'(T)&=12212\\
    \end{align*}
\end{example}

\subsection*{A Crystal Approach to the Kronecker Coefficient Problem}

The Kronecker coefficient problem is usually phrased as being about interpreting certain \emph{numbers} combinatorially.  However, in a deeper sense it is really about understanding the internal structure of certain \emph{representations} combinatorially.  One way to understand representations combinatorially is to use crystals,  this motivates the following refinement of the Kronecker coefficient problem:

\medskip

\noindent \textbf{Open Problem:} \emph{Construct a $\mathfrak{gl}_n \times \mathfrak{gl}_m$ crystal structure on the set of bitableaux of shape $\lambda$ with entries in $[n] \times [m]$.}

\medskip

We'll note that the existence of such a crystal structure is guaranteed purely for weight reasons, but we need an actual description of the crystal for this to be useful. Given such a crystal, the Kronecker coefficient $g(\lambda,\mu, \nu)$ will be the number of components of the crystal that are isomorphic to the simple $\mathfrak{gl}_n \times \mathfrak{gl}_m$ crystal $B(\mu) \times B(\nu)$, or equivalently the number of highest-weight elements of weight $(\mu, \nu)$. 

In practice, constructing a $\mathfrak{gl}_n$ or $\mathfrak{gl}_m$ crystal on a set of combinatorial objects typically involves extracting a reading word from each object, and the highest weight elements are those with a Yamanouchi reading word.  This is the motivation behind Conjecture \ref{mainconj}.

\section{The $\mathfrak{gl}_m$-Crystal Structure}\label{sec:glm-crystal}

In this section, we will give our main technical result:  We construct a $\mathfrak{gl}_m$-crystal structure on the set of bitableaux of shape $\lambda$ with entries in $[n] \times [m]$, such that the crystal operators preserve the $\mathfrak{gl}_n$-weights. 

Using this crystal structure, we are then able to read off the highest weight vectors to give a positive combinatorial formula for the monomial expansion of a Kronecker product of Schur functions.

\subsection{ The copies of $\mathfrak{gl}_n \times \mathfrak{gl}_m$ and $\mathfrak{gl}_m$ inside $\mathfrak{gl}_{nm}$}

Our overall strategy will be to start with a $\mathfrak{gl}_{nm}$-crystal on tableaux of shape $\lambda$ and use it to construct our desired $\mathfrak{gl}_m$-crystal on bitableaux of the same shape. To do this we will first recall how the relevant copies of $\mathfrak{gl}_n \times \mathfrak{gl}_m$ and $\mathfrak{gl}_m$ sit inside $\mathfrak{gl}_{nm}$.

We may identify $[n] \times [m]$ with $[nm]$ via the map $(i,j) \to (i-1)m+j$, which can be easily seen to identify the lexicographic order with the standard order on $[nm]$. This allows us to identify our collection of bitableaux of shape $\lambda$ with the collection of semistandard tableaux of the same shape with entries from $[nm]$, which we know how to equip with the structure of a $\mathfrak{gl}_{nm}$-crystal. 

Moreover, this identification $(i,j) \to (i-1)m+j$ is compatible with the embedding $\mathfrak{gl}_n \times \mathfrak{gl}_m \hookrightarrow \mathfrak{gl}_{nm}$ that sends a pair of matrices $(A,B)$ to the block matrix:

\[
A\hat{*}B = 
\begin{bmatrix}
a_{11}\mathbf{1}_m + B & a_{12}\mathbf{1}_m & \dots & a_{1n}\mathbf{1}_m\\
a_{21}\mathbf{1}_m & a_{22}\mathbf{1}_m + B & \dots & a_{2n}\mathbf{1}_m \\
\vdots & \vdots & \ddots & \vdots \\
a_{n1}\mathbf{1}_m & a_{n2}\mathbf{1}_m & \dots & a_{nn}\mathbf{1}_m + B
\end{bmatrix}
\]

Here the $a_{ij}$'s are the entries of $A$ and $\mathbf{1}_m$ denotes the $m \times m$ identity matrix.  The compatibility with our identification  $(i,j) \to (i-1)m+j$ can be seen by looking at the diagonal entries: $$a_{ii} + b_{jj} = (A\hat{*}B)_{(i-1)m+j,(i-1)m+j}$$

This formula for $A \hat{*} B$ may seem unfamiliar to some. We'll note that it is just the Lie algebra version (i.e. the map of tangent spaces at the identity element) of the more familiar formula for the embedding of Lie groups $GL_n \times GL_m \hookrightarrow GL_{nm}$ given by the Kronecker product of matrices.

\[
X*Y = 
\begin{bmatrix}
x_{11}Y & x_{12}Y & \dots & x_{1n}Y\\
x_{21}Y & x_{22}Y & \dots & x_{2n}Y \\
\vdots & \vdots & \ddots & \vdots \\
x_{n1}Y & x_{n2}Y & \dots & x_{nn}Y
\end{bmatrix}
\]

If we set $A = 0$ in the above formula, we see that the copy of $\mathfrak{gl}_m$ we are interested in is embedded inside $\mathfrak{gl}_{nm}$ diagonally as:

\[
B  \rightarrow 
\begin{bmatrix}
B & 0 & \dots &0\\
0 & B & \dots & 0 \\
\vdots & \vdots & \ddots & \vdots \\
0 & 0  & \dots & B
\end{bmatrix}
\]

Similarly, if we look at the image of the standard Cartan subalgebra of $\mathfrak{gl}_n$ (the diagonal matrices) we get:

\[
\begin{bmatrix}
x_1 & 0 & \dots &0\\
0 & x_2 & \dots & 0 \\
\vdots & \vdots & \ddots & \vdots \\
0 & 0  & \dots & x_n
\end{bmatrix}
\rightarrow 
\begin{bmatrix}
x_1 \mathbf{1}_m & 0 & \dots &0\\
0 & x_2 \mathbf{1}_m & \dots & 0 \\
\vdots & \vdots & \ddots & \vdots \\
0 & 0  & \dots & x_n \mathbf{1}_m
\end{bmatrix}
\]
In particular notice that this coincides with the center of $\mathfrak{gl}^n_m$,  so the $\mathfrak{gl}_n$-weights will just record the action of the center on the $\mathfrak{gl}_m$-representations (a.k.a. the degrees of those representations) when we restrict a representation from $\mathfrak{gl}_{nm}$.

\subsection{Description of the Construction}

We will now describe how to compute the restriction of a crystal for an irreducible representation from $\mathfrak{gl}_{nm}$ to this block diagonal copy of $\mathfrak{gl}_m$,  giving a $\mathfrak{gl}_m$-crystal structure on the set of bitableaux preserving the $\mathfrak{gl}_{n}$-weights.

This section will be a more representation theoretic explanation of what we are doing and why it works.  The following section will go into the combinatorics more explicitly, and explain what the raising and lowering operators actually do to a bitableau.  

\medskip

\noindent \textbf{Step 0: Start with a  $\mathfrak{gl}_{nm}$-crystal.} 

The process we are about to describe for restricting a $\mathfrak{gl}_{nm}$-crystal to the diagonal copy of $\mathfrak{gl}_{m}$ is a uniform  construction for all $\mathfrak{gl}_{nm}$-crystals.  However for our purposes we will fix it to be the standard $\mathfrak{gl}_{nm}$-crystal structure on the set of semistandard tableaux of shape $\lambda$ and entries in $[nm]$.

\medskip

\noindent \textbf{Step 1: Restrict to the block-diagonal copy of $\mathfrak{gl}_m^n$}.

Before restricting all the way to our diagonal copy of $\mathfrak{gl}_m$, we first restrict to an intermediate subalgebra, the block-diagonal copy of $\mathfrak{gl}_m^n$:

\[
\begin{bmatrix}
\mathfrak{gl}_m & 0 & \dots &0\\
0 & \mathfrak{gl}_m & \dots & 0 \\
\vdots & \vdots & \ddots & \vdots \\
0 & 0  & \dots & \mathfrak{gl}_m
\end{bmatrix}
\]
where we have $n$ independently chosen $m\times m$ matrices along the diagonal (as opposed to our copy of $\mathfrak{gl}_m$, where they are all equal to one another).  

Restricting a $\mathfrak{gl}_{nm}$-crystal to a standard Levi subalgebra such as this is easy to describe. To restrict to a standard Levi factor like this, one just keeps the same root operators $e_i$ and $f_i$ for those simple roots which belong to the Levi factor, and ignore those that do not.  Here is an example of a $\mathfrak{gl}_4$-crystal restricted to the standard Levi copy of $\mathfrak{gl}_2 \times \mathfrak{gl}_2$: 

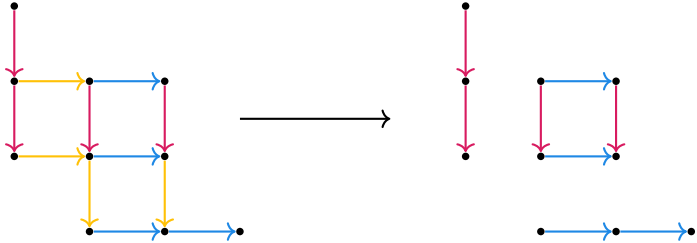
\begin{center}
\begin{tikzpicture}[
    every node/.style={circle, fill=black, inner sep=1pt},
    node distance=1cm
  ]
  \foreach \i/\x/\y in {
    1/0/0, 2/0/-1, 3/1/-1, 4/2/-1,
    5/0/-2, 6/1/-2, 7/2/-2,
    8/1/-3, 9/2/-3, 10/3/-3
  } \node (\i) at (\x,\y) {};

  \draw[c1, thick, ->]
    (1) -- (2);
    \draw[c1, thick, ->]
    (2) -- (5);
    \draw[c1, thick, ->]
    (3) -- (6);
    \draw[c1, thick, ->]
    (4) -- (7);
  \draw[c3, thick, ->]
    (2) -- (3);
    \draw[c3, thick, ->]
    (5) -- (6);
    \draw[c3, thick, ->]
    (6) -- (8);
    \draw[c3, thick, ->]
    (7) -- (9);
  \draw[c2, thick,->]
    (3) -- (4);
    \draw[c2, thick,->]
    (6) -- (7);
    \draw[c2, thick,->]
    (8) -- (9);
    \draw[c2, thick,->]
    (9) -- (10);

 \draw[->, thick]
  (3,-1.5) -- ++(2,0 ) ;

  \begin{scope}[xshift=6cm]
    \foreach \i/\x/\y in {
      1/0/0, 2/0/-1, 3/1/-1, 4/2/-1,
      5/0/-2, 6/1/-2, 7/2/-2,
      8/1/-3, 9/2/-3, 10/3/-3
    } \node (\i') at (\x,\y) {};

    \draw[c1, thick, ->]
      (1') -- (2');
      \draw[c1, thick, ->]
      (2') -- (5');
      \draw[c1, thick, ->]
      (3') -- (6');
      \draw[c1, thick, ->]
      (4') -- (7');
    \draw[c2, thick,->]
      (3') -- (4');
       \draw[c2, thick,->]
      (6') -- (7');
       \draw[c2, thick,->]
      (8') -- (9');
       \draw[c2, thick,->]
      (9') -- (10');
  \end{scope}
\end{tikzpicture}
\end{center}

Moreover the Cartan subalgebra for $\mathfrak{gl}_{n}$ is just scaling each of these factors by its degree as a $\mathfrak{gl}_{m}$ representation. So in particular the $\mathfrak{gl}_{m}$-operators all preserve the $\mathfrak{gl}_{n}$-weights.  Here are the $\mathfrak{gl}_{n}$-weights in the above example:

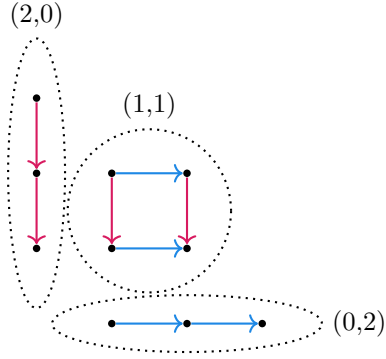
\begin{center}
\begin{tikzpicture}[
    vertex/.style={circle, fill=black, inner sep=1pt},
    node distance=1cm
  ]
  \foreach \i/\x/\y in {
    1/0/0,   2/0/-1, 3/1/-1, 4/2/-1,
    5/0/-2,  6/1/-2, 7/2/-2,
    8/1/-3,  9/2/-3, 10/3/-3
  } \node[vertex] (\i) at (\x,\y) {};

  \draw[c1,   thick,->] (1)--(2);
  \draw[c1,   thick,->] (2)--(5);
  \draw[c1,   thick,->](3)--(6);
  \draw[c1,   thick,->](4)--(7);
  \draw[c2,  thick,->] (3)--(4);
  \draw[c2,  thick,->](6)--(7);
  \draw[c2,  thick,->] (8)--(9);
  \draw[c2,  thick,->](9)--(10);

  \node[draw,ellipse, dotted, thick, fill=none, inner sep=6pt, fit=(1)(2)(5),
        label=above:{(2,0)}] {};
  \node[draw,ellipse, dotted, thick, fill=none, inner sep=6pt, fit=(3)(4)(6)(7),
        label=above:{(1,1)}] {};
  \node[draw,ellipse, dotted, thick, fill=none, inner sep=6pt, fit=(8)(9)(10),
        label=right:{(0,2)}] {};
\end{tikzpicture}

\end{center}

\medskip

\noindent \textbf{Step 2: Restrict from $\mathfrak{gl}_m^n$ to the diagonally embedded copy of $\mathfrak{gl}_m$.}  

Recall that if $\mathfrak{g}$ is a Lie algebra, and $V$ and $W$ are $\mathfrak{g}$-representations then $V \otimes W$ is naturally a $\mathfrak{g} \times \mathfrak{g}$-representation with $(x,y) \in \mathfrak{g} \times \mathfrak{g}$ acting via $$(x,y)\cdot (v \otimes w) = xv \otimes w + v \otimes yw$$

If we restrict this action to the diagonal embedding of $\mathfrak{g}$ inside $\mathfrak{g} \times \mathfrak{g}$ we get an action of $\mathfrak{g}$ on $V \otimes W$, where $x \in \mathfrak{g}$ acts via:
$$ x \cdot (v \otimes w) = xv\otimes w + v \otimes xw$$
which one can recognize as the usual way $\mathfrak{g}$ acts on a tensor product $V\otimes W$. 

Really this is just how one defines the tensor product of $\mathfrak{g}$ representations to begin with, but for our purposes it means understanding restriction from  $\mathfrak{gl}_m^n$ to $\mathfrak{gl}_m$ is the same thing as understanding $n$-fold tensor products of $\mathfrak{gl}_m$-representations.

In terms of crystals, each component of a $\mathfrak{gl}_m^n$-crystal is of the form $B(\mu_1) \otimes B(\mu_2) \otimes \dots \otimes B(\mu_n)$ where each $B(\mu_i)$ is an irreducible $\mathfrak{gl}_m$-crystal and each copy of $\mathfrak{gl}_m$ acts on a different factor independently.  In order to get our $\mathfrak{gl}_m$-crystal crystal we will just apply a crystal tensor product rule to $B(\mu_1) \otimes B(\mu_2) \otimes \dots \otimes B(\mu_n)$. 

Recasting this slightly: For any vertex in our crystal, and simple root $\alpha$ for $\mathfrak{gl}_m$ we have $n$ different crystal raising operators $e_\alpha^1, 
e_\alpha^2, \dots, e_\alpha^n$ -- one for each factor, and the same for the crystal lowering operators $f_\alpha^i$.  A tensor product rule is then a way of coherently choosing one of these $e_\alpha^i$ to act by at every point to construct a single raising operator $e_\alpha$ (and a single lowering operator $f_\alpha$) on the same underlying set. 

In particular note that the edges of the tensor product $\mathfrak{gl}_m$-crystal are a subset of the edges of the $\mathfrak{gl}_m^n$ crystal. Since this whole $\mathfrak{gl}_m^n$-crystal preserves the $\mathfrak{gl}_n$-weights so does the $\mathfrak{gl}_m$-crystal.

\medskip

\noindent \textbf{Recap:} Without all the exposition and explanations here is the construction:

\begin{enumerate}
\setcounter{enumi}{-1}
    \item Start with a $\mathfrak{gl}_{nm}$-crystal.

    \item Restrict it to $\mathfrak{gl}_{m}^n$-crystal by just ignoring some root operators.

    \item Use a crystal tensor product rule to restrict from $\mathfrak{gl}_{m}^n$ to the diagonal copy of $\mathfrak{gl}_{m}$.
\end{enumerate}

\subsection{A More Combinatorial Description}

The description in the last section gives a picture of what's going on and why this works from a representation theory perspective.  Now let's see what it is actually doing combinatorially in terms of bitableaux. 

\medskip

\noindent \textbf{Step 0: Start with a $\mathfrak{gl}_{nm}$-crystal.} 

In terms of the underlying set, this will just be the collection of semistandard Young tableaux of shape $\lambda$ and entries in $[nm]$.  Using our bijection $(i-1)m+j \leftrightarrow (i,j)$ we identify this with the collection of bitableaux with entries in $[n] \times [m]$.  Here is an example of a tableau with entries in $[6]$ and its corresponding bitableau with entries in $[3] \times [2]$:

$$
   \ytableausetup
{boxsize=2em}
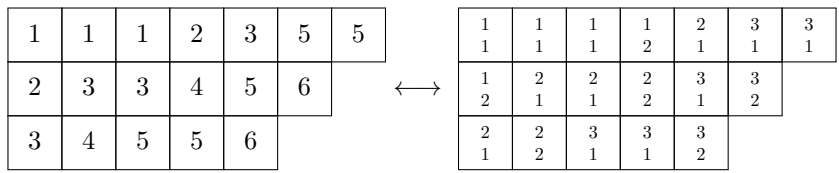
\begin{ytableau}
            1 & 1 & 1 & 2 & 3 & 5 & 5\\
            2 & 3 &3 & 4 & 5 & 6\\
            3 & 4 &5 &5 &6 
        \end{ytableau}
        \longleftrightarrow \ 
        \bitableau{
            \bi11 & \bi11 & \bi11 & \bi12 & \bi21 & \bi31 & \bi31 \\
            \bi12 & \bi21 & \bi21 & \bi22 & \bi31 & \bi32 \\
            \bi21 & \bi22 & \bi31 & \bi31 & \bi32
        }
$$

One can also trace through what this does to the crystal operators: For example, the operator that turns an $(i-1)m+j$ into an $(i-1)m+j+1$ (with $j<m$)  will now turn an  $\bitableau{
            \bi ij}$ into an $\bitableau{
            \bi i{j+1}}$.

\medskip 

\noindent \textbf{Step 1: Restrict to the block diagonal copy of $\mathfrak{gl}_m^n$}. 

Combinatorially this just means separating a bitableau into a collection of tableaux of skew-shape,  with each skew tableau corresponding to a fixed top entry:
$$
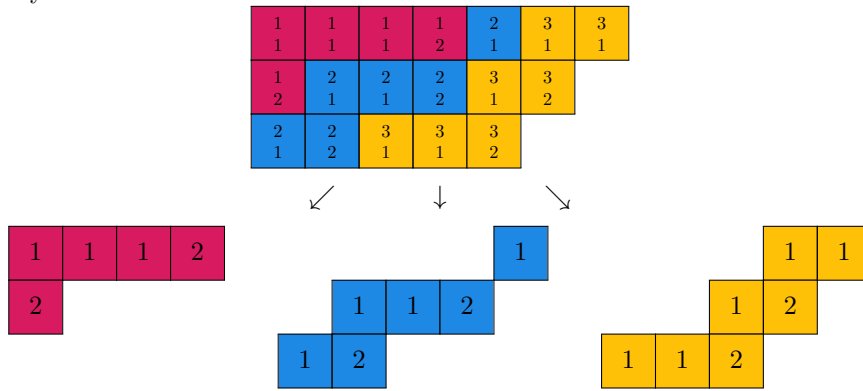
\bitableau{
            *(c1)\bi11 & *(c1) \bi11 & *(c1)\bi11 &  *(c1)\bi12 & *(c2) \bi21 & *(c3)\bi31 & *(c3)\bi 31 \\
            *(c1)\bi12 & *(c2) \bi21 & *(c2) \bi21 & *(c2) \bi22 & *(c3) \bi31 & *(c3) \bi32 \\
           *(c2) \bi21 & *(c2) \bi22 & *(c3) \bi31 & *(c3) \bi31 & *(c3) \bi32 
        }
    $$
 $$ \swarrow \hspace{1.3cm} \downarrow \hspace{1.3cm} \searrow $$
    $$
    \ytableausetup
{boxsize=2em}
\begin{ytableau}
             *(c1) 1 &  *(c1) 1 &  *(c1) 1 &  *(c1) 2 \\
             *(c1) 2 \\
             \none 
        \end{ytableau}
        \hspace{.7cm}
       \begin{ytableau}
             \none &  \none &  \none  &  \none & *(c2) 1 \\
             \none & *(c2) 1 & *(c2) 1 & *(c2) 2 \\
             *(c2) 1 & *(c2) 2
        \end{ytableau} 
 \hspace{.7cm}
 \begin{ytableau}
             \none &  \none &  \none  & *(c3) 1 & *(c3) 1 \\
             \none & \none & *(c3) 1 &  *(c3) 2 \\
             *(c3) 1 &  *(c3) 1 &  *(c3) 2
        \end{ytableau} 
$$

  Note that the crystal operators that are in $\mathfrak{gl}_{nm}$ but not in $\mathfrak{gl}_m^n$ are exactly those that go between different color classes. The $\mathfrak{gl}_m^n$ crystal structure we get from restriction is just made up of the separate, independent $\mathfrak{gl}_m$-crystal structures on each of these skew-shapes.

In particular we can just read off the reading words from each of these skew-shape tableaux and the raising and lowering operators on each piece will be determined by the usual parenthesization rule. 

\medskip 

\noindent \textbf{Step 2: Restrict from $\mathfrak{gl}_m^n$ to the diagonally embedded copy of $\mathfrak{gl}_m$.}

Now we take these $\mathfrak{gl}_m^n$-crystals and combine them into a single $\mathfrak{gl}_m$ via the tensor product rule.  One way to do this explicitly is to take the reading words for the separate skew-shapes and concatenate them:

  $$
    \ytableausetup
{boxsize=2em}
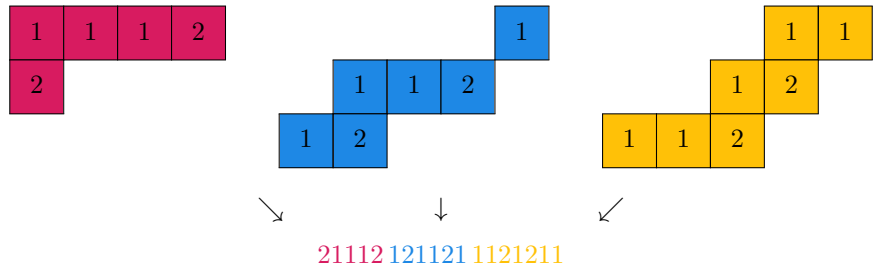
\begin{ytableau}
             *(c1) 1 &  *(c1) 1 &  *(c1) 1 &  *(c1) 2 \\
             *(c1) 2 \\
             \none 
        \end{ytableau}
        \hspace{.7cm}
       \begin{ytableau}
             \none &  \none &  \none  &  \none & *(c2) 1 \\
             \none & *(c2) 1 & *(c2) 1 & *(c2) 2 \\
             *(c2) 1 & *(c2) 2
        \end{ytableau} 
 \hspace{.7cm}
 \begin{ytableau}
             \none &  \none &  \none  & *(c3) 1 & *(c3) 1 \\
             \none & \none & *(c3) 1 &  *(c3) 2 \\
             *(c3) 1 &  *(c3) 1 &  *(c3) 2
        \end{ytableau} 
$$
 $$ \searrow \hspace{2cm} \downarrow \hspace{2cm} \swarrow $$
$$ \textcolor{c1}{21112}  \,\textcolor{c2}{121121} \,\textcolor{c3}{1121211} $$

We then can tell how to perform a raising or lowering operator by performing the parenthesization rule to this combined reading word. 

The fact that the crystal operators preserve the $\mathfrak{gl}_n$-weights corresponds to the fact that these raising and lowering operators only change the numbers inside the boxes, and don't change the number of boxes of each color. 

\subsection{Streamlined Combinatorics}

This process of separating the bitableau into tableau of skew shapes, and then concatenating the reading words makes it conceptually clear what we are doing.  However in practice we usually just extract the reading word directly: 

$$
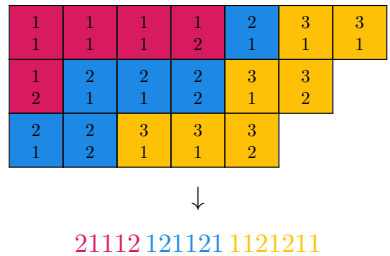
\bitableau{
            *(c1)\bi11 & *(c1) \bi11 & *(c1)\bi11 &  *(c1)\bi12 & *(c2) \bi21 & *(c3)\bi31 & *(c3)\bi 31 \\
            *(c1)\bi12 & *(c2) \bi21 & *(c2) \bi21 & *(c2) \bi22 & *(c3) \bi31 & *(c3) \bi32 \\
           *(c2) \bi21 & *(c2) \bi22 & *(c3) \bi31 & *(c3) \bi31 & *(c3) \bi32 
        }
    $$
     $$ \downarrow $$
$$ \textcolor{c1}{21112}\,  \textcolor{c2}{121121}\, \textcolor{c3}{1121211} $$

To do this we read off the bottom entries  of those boxes where the top entry is $1$ (left to right in each row, bottom to top), followed by the bottom entries of the boxes with top entry $2$ (again left to right in each row, bottom to top), and so on.  This is the reading word $w(T)$ defined in Section \ref{BitabIntro}.

To see how a raising or lowering operator acts we perform the parenthesization procedure to this word:
$$ \textcolor{c1}{()))(} \, \textcolor{c2}{)())()} \,\textcolor{c3}{))()())} $$
in this example we see that every ``(" is matched with a ``)" so this is a Yamanouchi word, meaning it corresponds to a highest weight vector in our crystal (so we do not change a $2$ to a $1$ from here). If we want to apply a lowering operator, we see that the last $1$ corresponds to the rightmost unmatched ``)" so our lowering operator turns this bitableau into: 

$$
\bitableau{
            *(c1)\bi11 & *(c1) \bi11 & *(c1)\bi11 &  *(c1)\bi12 & *(c2) \bi21 & *(c3)\bi31 & *(c3) \boldbi 32 \\
            *(c1)\bi12 & *(c2) \bi21 & *(c2) \bi21 & *(c2) \bi22 & *(c3) \bi31 & *(c3) \bi32 \\
           *(c2) \bi21 & *(c2) \bi22 & *(c3) \bi31 & *(c3) \bi31 & *(c3) \bi32 
        }
    $$

    We'll  note that we could also have used the reading word $w'(T)$, which corresponds to using a different tensor product rule.  The crystal operators obtained in this way would be different, but the crystals overall would be isomorphic.

    \subsection{The Monomial Expansion of the Kronecker Product}

    Now that we have this construction, we are ready to restate and prove our combinatorial interpretation for the monomial expansion of the Kronecker product of Schur polynomials. 

    \begin{customtheorem}{2} The Kronecker product of two Schur functions can be expanded in the monomial basis as:
    \begin{align*}
        s_\lambda\ast s_\nu&=\sum_{\mu}d(\lambda,\mu,\nu)m_\mu.
    \end{align*}
    where $d(\lambda,\mu,\nu)$ counts the number of bitableaux $T$ of shape $\lambda$, with weights $a(T)= \mu$ and $b(T)= \nu$, and with Yamanouchi sort-by-top reading word $w(T)$. 
\end{customtheorem}

\begin{proof}
    The Kronecker product is defined as the graded dual of the Kronecker coproduct $\Lambda \to \Lambda \otimes \Lambda$ given by $p \to p[\mathbf{xy}]$. In the Schur basis, the rule is given by
    $$ s_\lambda[\mathbf{xy}] = \sum_{\tau, \nu}g(\lambda,\tau,\nu) s_\tau[\mathbf{x}]s_\nu[\mathbf{y}]. $$
If we interpret the $\mathbf{x}$ variables as the eigenvalues of a matrix in $GL(V)$ and the $\mathbf{y}$ variables as the eigenvalues of a matrix in $GL(W)$, then this is decomposing the character of $S^{\lambda}(V \otimes W)$ as a $GL(V) \times GL(W)$ representation.  

The monomial expansion for the Kronecker product corresponds to expanding the right-hand side of the Kronecker coproduct in terms of the monomial basis for the $\mathbf{x}$ variables, but still in the Schur basis for the $\mathbf{y}$ variables:

$$ s_\lambda[\mathbf{xy}] = \sum_{\mu, \nu} d(\lambda,\mu, \nu) m_\mu[\mathbf{x}]s_\nu[\mathbf{y}] $$
In terms of representation theory, the right hand side is decomposing $S^{\lambda}(V \otimes W)$ as a $T_V \times GL(W)$ representation, where $T_V \subset GL(V)$ is a maximal torus.

This is exactly what our crystal is doing: 
decomposing a $GL(V \otimes W) \cong GL_{nm}$  representation as a $GL(W) \cong GL_m$-representation which preserves the $GL(V) = GL_n$-weights. The coefficient of $m_\mu[\mathbf{x}]s_\nu[\mathbf{y}]$ in this expansion is equal to the multiplicity of $S^{\nu}(W)$ in the $GL_n$-weight space for $\mu$. This is equal to the number of $GL_m$ highest weight vectors of $GL_m$-weight $\nu$ and $GL_n$-weight $\mu$. Our crystal structure is defined using the sort-by-top reading word, and being a highest weight vector exactly corresponds to the reading word being Yamanouchi. 
\end{proof}

\section{Known cases}\label{sec:known}

In a handful of cases, Kronecker products are well-understood. We will show now that our framework generalizes many of these known cases in a uniform way, reinforcing our belief that this is a good place to look for a solution to the Kronecker coefficient problem.

\subsection{RSK and dual RSK crystals}

 To put RSK and dual RSK into our framework, we first introduce two specialized methods for extracting a reading word from the top entries of a tableau. Let
 \begin{itemize}
    \item $u(T)$ be the word obtained by reading (in the usual left to right bottom to top way) the top entries of boxes with bottom entry one, then the top entries of boxes with bottom entry two, and so on; and
    \item $u'(T)$ be the word obtained by similarly reading the top entries of boxes with the largest bottom entry, then the top entries of boxes with the second largest bottom entry, and so on.
\end{itemize}

Let \[\RSK:\BB_{(r)}\to \bigcup_{\lambda\vdash r}B(\lambda)\times B(\lambda)\] be the bijection that considers a one-row bitableau as a biword and performs RSK on it.

\begin{proposition}
    Let $T\in \BB_{(r)}$. Then,   \[\RSK(T)=(P(w(T)),P(u(T))).\]
\end{proposition}

\begin{proof}
    This result follows immediately from the symmetry of the RSK algorithm upon exchanging the rows of a biword.
\end{proof}

\begin{corollary}
    If $\BB_{(r)}$ is equipped with the crystal structures applying the parenthesization operations to the reading words $w(T)$ and $u(T)$, then $\RSK$ is an isomorphism of $\mathfrak{gl}_n\times\mathfrak{gl}_m$-crystals.
\end{corollary}

The analogous algorithm for the column shape turns out to be Burge insertion. A Burge word is like a biword except for two modifications: when two top entries coincide, the columns are sorted in \emph{decreasing} order by second entry, and no column may be repeated. Burge insertion is then carried out by row-inserting the second row of the Burge word, and recording the corresponding entries in the first row. This process results in a recording tableau that is row, rather than column, strict, so to define a map to a pair of semistandard tableaux, we end by taking the transpose of the recording tableau.

We define a map \[\bRSK:\BB_{(1^r)}\to\bigcup_{\lambda\vdash r} B(\lambda)\times B(\lambda')\] by first reading entries of the form $\bi1\ast$, then $\bi2\ast$, and so on from bottom to top to form a Burge word, then performing Burge insertion on this word.

\begin{example}
    Let \[T=\bitableau{
                \bi11\\
                \bi12\\
                \bi21\\
                \bi23\\
                \bi24\\
                \bi31\\
                \bi33
            }\in\BB_{(1^7)}.\] The corresponding Burge word is \begin{align*}
                \left(\bi12\bi11\bi24\bi23\bi21\bi33\bi31\right).
            \end{align*} Applying Burge insertion and transposing the second tableau, we get \begin{align*}
                \bRSK(T)=\left(\begin{ytableau}
                    1 & 1 & 1\\
                    2 & 3 & 3\\
                    4
                \end{ytableau},\begin{ytableau}
                    1 & 1 & 2\\
                    2 & 2\\
                    3 & 3
                \end{ytableau}\right).
            \end{align*}
\end{example}

We will need one more variant of RSK: dual RSK, which we write as \[\RSK'(w)=(P'(w),Q'(w))\] (See \cite[Appendix A]{fulton1997young}). We are only interested in leveraging relationships between dual RSK and Burge insertion, so for our purposes, it is only important that $\RSK'$ takes in a word and outputs a row-strict tableau (i.e. one whose transpose is semistandard).

\begin{proposition}
    Let $T\in \BB_{(1^r)}$. Then,   \[\bRSK(T)=(P(w(T)),P(u'(T))).\]
\end{proposition}

\begin{proof}
    First, consider the Burge word $\left(\begin{matrix}
        \ast\\w(T)
    \end{matrix}\right)$ obtained by reading the entries of $T$ by their top entry, low to high, and let $\bRSK(T)=(P,Q)$. By definition, $P=P(w(T))$.

    Note that if we swap the rows of this Burge word and sort it into a biword, we get something of the form $\left(\begin{matrix}
        \ast\\rev(u'(T))
    \end{matrix}\right)$. By \cite[p.~200]{fulton1997young}, we have that $P'(rev(u'(T)))=Q^t$. Finally, by \cite[Prop.~2.3.14]{butler1994subgroup}, we have that \begin{align*}
        P'(rev(u'(T)))&=Q^t\\
        &\Updownarrow\\
        P(rev(rev(u'(T))))&=(Q^t)^t\\
        &\Updownarrow\\
        P(u'(T))&=Q.
    \end{align*}

\end{proof}

\begin{corollary}
    If $\BB_{(1^r)}$ is equipped with the crystal structures applying the parenthesization operations to the reading words $w(T)$ and $u'(T)$, then $\bRSK$ is an isomorphism of $\mathfrak{gl}_n\times\mathfrak{gl}_m$-crystals.
\end{corollary}

The perspective in this section suggests a strategy: find a way to extract a second reading word from a bitableau in a way compatible with $w(T)$.

\medskip

\noindent \textbf{Open Problem:} \emph{Generalize $u(T)$ and $u'(T)$ to a reading method for all shapes that gives a $\mathfrak{gl}_n$-crystal structure commuting with the $\mathfrak{gl}_m$-crystal structure afforded by $w(T)$.}

\subsection*{An Insertion Version of the Kronecker Coefficient Problem}

Above we used the connection between crystals and RSK insertion to put RSK and dual RSK into our crystal-based approach to the Kronecker coefficient problem. It is also possible to remove all mentions of crystals entirely and formulate a version of the Kronecker coefficient problem entirely in terms of RSK insertion.\\

\noindent \textbf{Open Problem:} \emph{Find an insertion algorithm that takes as input a bitableaux $B$ and outputs a pair of tableaux $(T,T')$ of potentially different shapes such that:}

\begin{itemize}
\item \emph{The map preserves the weights: $w(T) = a(B)$ and $w(T') = b(B)$.}

\item \emph{Given any three partitions $\lambda, \mu, \nu$. All pairs of tableaux of shapes $(\mu, \nu)$ have the same number of preimages of shape $\lambda$.}
\end{itemize}

We note that this is actually slightly weaker than the crystal formulation, but still enough to get a combinatorial interpretation of the Kronecker coefficients. Given such an insertion algorithm, the Kronecker coefficient $g(\lambda,\mu,\nu)$ will be the number of times a pair of tableaux of shapes $(\mu,\nu)$ occurs as output with input of shape $\lambda$.

\subsection{Kronecker Tableaux}

In \cite{ballantine2005combinatorial}, the authors introduce Kronecker tableaux in order to give a combinatorial interpretation for Kronecker coefficients involving a two-part shape $(n-p,p)$ and a partition $\lambda\vdash n$ whenever $\lambda_1\geq 2p-1$ or $\ell(\lambda)\geq 2p-1$. We now give a definition of their Kronecker tableaux cast in the language of bitableaux.

First, let $\BB'_\lambda(2,m)$ be the subset consisting of $T\in\BB_\lambda(2,m)$ with $w'(T)$ Yamanouchi (that is, the highest-weight vectors in the corresponding $\mathfrak{gl}_m$-crystal).

\begin{definition}\label{def:kronecker_tab}
    A Kronecker tableau of shape $\lambda$ is a lexicographic bitableau $T\in\BB'_\lambda(2,m)$ satisfying the following condition where $\alpha=(\alpha_1,\alpha_2,\ldots)$ is the shape formed by the boxes of $T$ with first entry one: \begin{itemize}
        \item[(I)] $\alpha_1=\alpha_2$, or
        \item[(II)] $\alpha_1>\alpha_2$ and any one of the following conditions are satisfied: \begin{itemize}
            \item[(i)] The number of \;$\bitableau{\bi21}$ in the second row of $T$ is exactly $\alpha_1-\alpha_2$.
            \item[(ii)] The number of \;$\bitableau{\bi22}$ in the first row of $T$ is exactly $\alpha_1-\alpha_2$.
        \end{itemize}
    \end{itemize}
\end{definition} 

\begin{remark}
    To obtain a Kronecker tableau as described in \cite{ballantine2005combinatorial} from a bitableau as in Definition \ref{def:kronecker_tab}, remove all boxes of the form $\bitableau{\bi1a}$ and remove the upper twos from the remaining boxes.
\end{remark}

In this language, they interpret the Kronecker coefficient $g(\lambda,(n-p,p),\nu)$ when $\lambda_1\geq 2p-1$ as the number of Kronecker tableaux $T$ of shape $\lambda$ with $a(T)=(p,n-p)$ and $b(T)=\nu$.

\newcommand{\gena}{*(c1)\bi{1}{*}}
\newcommand{\genb}{*(c2)\bi{2}{*}}
\newcommand{\genc}{*(c3)\bi{2}{*}}

In the remainder of this section, we explain how the characterization in Definition \ref{def:kronecker_tab} is explained by a map $\varphi:\BB'_\lambda(2,m)\to\BB'_\lambda(2,m)\cup\{0\}$ that lowers the $a$-weight, which we conjecture can be completed to a lowering operation for a $\mathfrak{gl}_2$-crystal structure on $\BB'_\lambda(2,m)$. That is, Kronecker tableaux when $\lambda_1\geq 2p-1$ appear to be highest-weight with respect to our $\mathfrak{gl}_m$-crystal and lowest-weight with respect to a $\mathfrak{gl}_2$-crystal, which we partially describe here.

Given a tableau $T\in \BB'_\lambda(2,m)$ such as \begin{align*}
    T=\bitableau{\gena&\gena&\gena&\gena&\bi{1}{a}&\genb&\genb&\genb&\genb\\
    \gena&\gena&\gena&\genc&\genc&\genc&\genc&\genc\\
    \gena&\gena&\genc&\genc\\
    \genc&\genc&\genc&\genc}
\end{align*} we obtain a new tableau $T'$ by changing the right-most $\bitableau{\bi1a}$ in the first row to a $\bitableau{\bi2a}$. \begin{align*}
    T'=\bitableau{\gena&\gena&\gena&\gena&\bi{2}{a}&\genb&\genb&\genb&\genb\\
    \gena&\gena&\gena&\genc&\genc&\genc&\genc&\genc\\
    \gena&\gena&\genc&\genc\\
    \genc&\genc&\genc&\genc}.
\end{align*} Let \[\varphi(T)=\begin{cases}
    T' & T'\in\BB'_\lambda(2,m)\\
    0 & \text{otherwise}
\end{cases}.\]

\begin{proposition}
    For $T\in\BB'_\lambda(2,m)$, we have that $\varphi(T)=0$ if and only if $T$ is a Kronecker tableau.
\end{proposition}

\begin{proof}

First, note that in order for $T'$ to be a bitableau at all, we must have that $\alpha_1>\alpha_2$ (i.e. condition (I) must not be met).

Now we split the reading word $w'(T)$ into subwords as follows. Let \begin{itemize}
    \item $p$ be the subword of the bottom of entries $\bi{1}{\ast}$,
    \item $b$ be the subword of the bottom of the entries $\bi{2}{\ast}$ in the first row, and
    \item $y$ be the subword of the bottom of the remaining entries $\bi{2}{\ast}$
\end{itemize} (highlighted in the above tableau $T$ as pink, blue, and yellow boxes respectively).

Assume that $w'(T)=ybpa$ is Yamanouchi. Note that we must then have $a=1$. In order for $T'$ to be column-strict, we must have that the box below the white box is not $\bitableau{\bi21}$ . Hence, the number of instances of $\bitableau{\bi21}$ in the second row must be strictly less than $\alpha_1-\alpha_2$ (i.e. condition (II)(i) must not be met).

We have that $\varphi(T)\neq0$ only if $w'(T')=y1bp$ is also Yamanouchi. The only way this new word fails to be Yamanouchi is if there is now a suffix with more twos than ones. Because $b$ is an increasing sequence and $p$ is a decreasing sequence, this means we need \begin{align*}
    m_2(b)+m_2(p)\leq m_1(p)
\end{align*} where $m_i(w)$ is the multiplicity of the letter $i$ in the word $w$. Note that because $w'(T)$ is Yamanouchi, $m_1(p)=\alpha_1-1$ and $m_2(p)=\alpha_2$. The above condition then becomes \begin{align*}
    m_2(b)<\alpha_1-\alpha_2.
\end{align*}

That is, condition (II)(ii) must not be met.

Hence, $\varphi(T)\in\BB'_\lambda(2,m)$ exactly when one of the conditions of $T$ being a Kronecker tableau fails.
\end{proof}

By \cite[Corollary 3.3]{ballantine2005combinatorial}, the Kronecker coefficient $g(\lambda,(n-p,p),\nu)$ when $\lambda_1<2p-1$ is bounded above by the number of Kronecker tableaux of shape $\lambda$ with $a(T)=(p,n-p)$ and $b(T)=\nu$. This supports the conjecture that $\varphi$ can be completed to a lowering operation $f^1$ that picks out which Kronecker tableaux should be considered lowest weight. 

\begin{example} The Kronecker coefficient $g((4,3),(4,3),(3,2,2))=1$, but there are two Kronecker tableaux of shape $(4,3)$, $a$-weight $(3,4)$ and $b$-weight $(3,2,2)$ shown below.
    \[\begin{array}{cc}
        \bitableau{\bi11\amp\bi11\amp\bi21\amp\bi22\\
                   \bi12\amp\bi23\amp\bi23} & \bitableau{\bi11\amp\bi11\amp\bi22\amp\bi23\\
                   \bi12\amp\bi21\amp\bi23}
    \end{array}\]

    One of these two tableaux should then not be lowest weight and should instead transform via $f^1$ into the following unique tableau of shape $(4,3)$ with $a$-weight $(2,5)$, $b$-weight $(3,2,2)$ and Yamanouchi reading word $w'(T)$. \[\bitableau{\bi11\amp\bi11\amp\bi22\amp\bi22\\
                   \bi21\amp\bi23\amp\bi23}\]
\end{example}

\vspace{.08in}

\noindent \textbf{Open Problem:} \emph{ Complete this $\mathfrak{gl}_2$-crystal structure to finish the interpretation of Kronecker coefficients with one partition having two parts.}

\begin{remark}
    The case of a two-part partition fits well into our framework because the set of two-part $a$-weights is naturally closed under the $\mathfrak{gl}_2$ raising and lowering operations. While combinatorial interpretations for Kronecker coefficients in the case of one hook shape are known (\cite{blasiak2016Kronecker,liu2017simplified}), they generally set the content of their tableaux to a hook shape. Because hook contents are not closed under a natural set of crystal operations, we expect that our crystal framework would better handle this case if one were to fix the \emph{shape} of the tableaux to a hook.
\end{remark}

\section{Examples} \label{sec:examples}

In this section, we give examples of what these bicrystals on bitableaux could look like for shapes other than a row or column.

\subsection{Maximum entry two}

First, we will look at some $\mathfrak{gl}_2 \times \mathfrak{gl}_2$-crystals on sets of bitableaux with entries in $[2] \times [2]$.  The second copy of $\mathfrak{gl}_2$ acts on the bottom weights, and we will use the crystal structure in Section \ref{sec:glm-crystal} with the lowering operators drawn as horizontal arrows.  

The crystal structure for first copy of $\mathfrak{gl}_2$ is then only partially determined by the condition that the two structures commute. The tableaux whose positions are determined by these constraints are given within the crystal structure while the undetermined pieces are given separately. We encourage the reader to cut out and paste these undetermined pieces at the appropriate weights to form their own conjectured bitableaux crystals.

\begin{example} The portion of the shape $(2,2)$ crystal that is determined by the constraints is shown below.

\vspace{.1in}

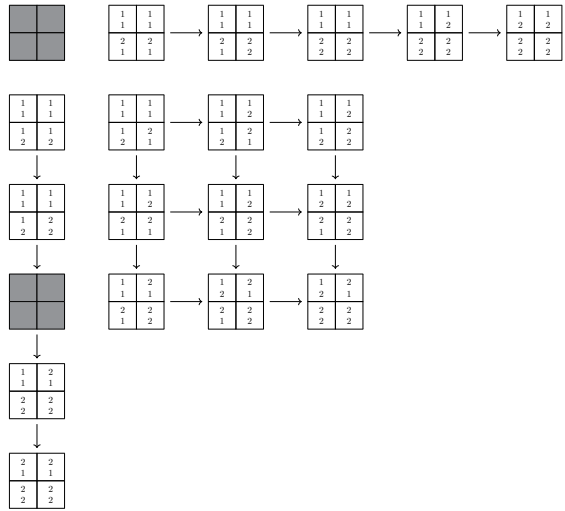
\adjustbox{scale=.5,center}{\begin{tikzcd}
         {\bitableau{\blankbox\amp\blankbox\\\blankbox\amp\blankbox}}    & {\bitableau{\bi11\amp\bi11\\\bi21\amp\bi21}} \arrow[r]           & {\bitableau{\bi11\amp\bi11\\\bi21\amp\bi22}} \arrow[r]           & {\bitableau{\bi11\amp\bi11\\\bi22\amp\bi22}} \arrow[r] & {\bitableau{\bi11\amp\bi12\\\bi22\amp\bi22}} \arrow[r] & {\bitableau{\bi12\amp\bi12\\\bi22\amp\bi22}} \\
{\bitableau{\bi11\amp\bi11\\\bi12\amp\bi12}} \arrow[d] & {\bitableau{\bi11\amp\bi11\\\bi12\amp\bi21}} \arrow[r] \arrow[d] & {\bitableau{\bi11\amp\bi12\\\bi12\amp\bi21}} \arrow[r] \arrow[d] & {\bitableau{\bi11\amp\bi12\\\bi12\amp\bi22}} \arrow[d] &              &    \\
{\bitableau{\bi11\amp\bi11\\\bi12\amp\bi22}} \arrow[d] & {\bitableau{\bi11\amp\bi12\\\bi21\amp\bi21}} \arrow[d] \arrow[r] & {\bitableau{\bi11\amp\bi12\\\bi21\amp\bi22}} \arrow[r] \arrow[d] & {\bitableau{\bi12\amp\bi12\\\bi21\amp\bi22}} \arrow[d] &              &    \\
{\bitableau{\blankbox\amp\blankbox\\\blankbox\amp\blankbox}} \arrow[d] & {\bitableau{\bi11\amp\bi21\\\bi21\amp\bi22}} \arrow[r]           & {\bitableau{\bi12\amp\bi21\\\bi21\amp\bi22}} \arrow[r]           & {\bitableau{\bi12\amp\bi21\\\bi22\amp\bi22}}           &              &    \\
{\bitableau{\bi11\amp\bi21\\\bi22\amp\bi22}} \arrow[d]           &                        &                        &              &              & \\
{\bitableau{\bi21\amp\bi21\\\bi22\amp\bi22}} & & & & &
\end{tikzcd}}

\vspace{.1in}
    
\noindent The undetermined gray pieces are filled with the following two tableaux.

\vspace{.1in}

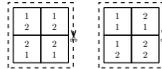
\adjustbox{scale=.5,center}{\begin{tabular}{cc}
     \begin{tikzpicture}
\node[rectangle,minimum width=.5in] (m) {\begin{minipage}{.58in}\bitableau{\bi12\amp\bi12\\\bi21\amp\bi21}\end{minipage}};
\draw[dashed] (m.south west) rectangle (m.north east);
\node[rotate=90] at (.35in,0) {\ding{34}};
\end{tikzpicture} & \begin{tikzpicture}
\node[rectangle,minimum width=.5in] (m) {\begin{minipage}{.58in}\bitableau{\bi11\amp\bi21\\\bi12\amp\bi22}\end{minipage}};
\draw[dashed] (m.south west) rectangle (m.north east);
\node[rotate=90] at (.35in,0) {\ding{34}};
\end{tikzpicture}
\end{tabular}}
Either way of placing them into the gray slots yields a valid $\mathfrak{gl}_2 \times \mathfrak{gl}_2$-crystal.  We believe one way looks more natural than the other, but we will try not to bias you.  A hypothetical solution to the crystal version of the Kronecker coefficient problem would pick one of them out in a consistent manner.

\end{example}

\begin{example} The portion of the shape $(3,1)$ crystal that is determined by the constraints is shown below.

\vspace{.1in}

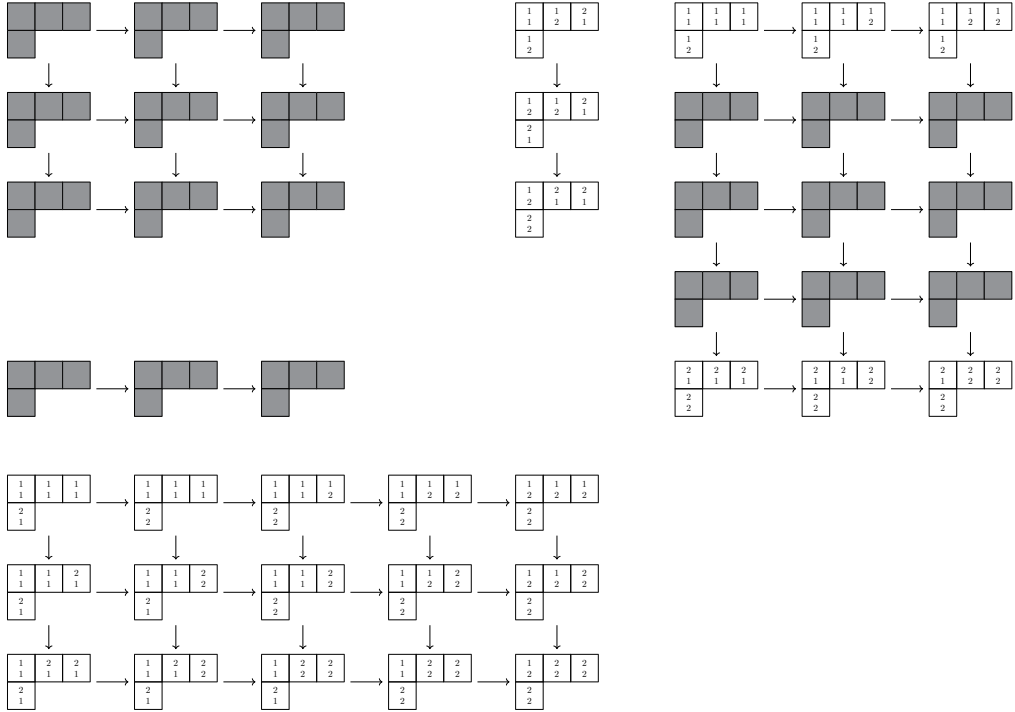
\adjustbox{scale=.5,center}{%
\begin{tikzcd}
{\bitableau{\blankbox\amp\blankbox\amp\blankbox\\\blankbox}} \arrow[r] \arrow[d] & {\bitableau{\blankbox\amp\blankbox\amp\blankbox\\\blankbox}} \arrow[r] \arrow[d] & {\bitableau{\blankbox\amp\blankbox\amp\blankbox\\\blankbox}} \arrow[d]           &                        & {\bitableau{\bi11\amp\bi12\amp\bi21\\\bi12}} \arrow[d] &  & {\bitableau{\bi11\amp\bi11\amp\bi11\\\bi12}} \arrow[d] \arrow[r] & {\bitableau{\bi11\amp\bi11\amp\bi12\\\bi12}} \arrow[r] \arrow[d] & {\bitableau{\bi11\amp\bi12\amp\bi12\\\bi12}} \arrow[d] \\
{\bitableau{\blankbox\amp\blankbox\amp\blankbox\\\blankbox}} \arrow[d] \arrow[r]          & {\bitableau{\blankbox\amp\blankbox\amp\blankbox\\\blankbox}} \arrow[d] \arrow[r]          & {\bitableau{\blankbox\amp\blankbox\amp\blankbox\\\blankbox}} \arrow[d]           &                        & {\bitableau{\bi12\amp\bi12\amp\bi21\\\bi21}} \arrow[d] &  & {\bitableau{\blankbox\amp\blankbox\amp\blankbox\\\blankbox}} \arrow[d] \arrow[r] & {\bitableau{\blankbox\amp\blankbox\amp\blankbox\\\blankbox}} \arrow[r] \arrow[d] & {\bitableau{\blankbox\amp\blankbox\amp\blankbox\\\blankbox}} \arrow[d] \\
{\bitableau{\blankbox\amp\blankbox\amp\blankbox\\\blankbox}} \arrow[r]           & {\bitableau{\blankbox\amp\blankbox\amp\blankbox\\\blankbox}} \arrow[r]           & {\bitableau{\blankbox\amp\blankbox\amp\blankbox\\\blankbox}}                     &                        & {\bitableau{\bi12\amp\bi21\amp\bi21\\\bi22}}           &  & {\bitableau{\blankbox\amp\blankbox\amp\blankbox\\\blankbox}} \arrow[d] \arrow[r] & {\bitableau{\blankbox\amp\blankbox\amp\blankbox\\\blankbox}} \arrow[r] \arrow[d] & {\bitableau{\blankbox\amp\blankbox\amp\blankbox\\\blankbox}} \arrow[d] \\
                       &                        &                        &                        &              &  & {\bitableau{\blankbox\amp\blankbox\amp\blankbox\\\blankbox}} \arrow[d] \arrow[r] & {\bitableau{\blankbox\amp\blankbox\amp\blankbox\\\blankbox}} \arrow[r] \arrow[d] & {\bitableau{\blankbox\amp\blankbox\amp\blankbox\\\blankbox}} \arrow[d] \\
{\bitableau{\blankbox\amp\blankbox\amp\blankbox\\\blankbox}} \arrow[r]           & {\bitableau{\blankbox\amp\blankbox\amp\blankbox\\\blankbox}} \arrow[r]           & {\bitableau{\blankbox\amp\blankbox\amp\blankbox\\\blankbox}}                     &                        &              &  & {\bitableau{\bi21\amp\bi21\amp\bi21\\\bi22}} \arrow[r]           & {\bitableau{\bi21\amp\bi21\amp\bi22\\\bi22}} \arrow[r]           & {\bitableau{\bi21\amp\bi22\amp\bi22\\\bi22}}           \\
                       &                        &                        &                        &              &  &                        &                        &              \\
{\bitableau{\bi11\amp\bi11\amp\bi11\\\bi21}} \arrow[r] \arrow[d] & {\bitableau{\bi11\amp\bi11\amp\bi11\\\bi22}} \arrow[r] \arrow[d] & {\bitableau{\bi11\amp\bi11\amp\bi12\\\bi22}} \arrow[r] \arrow[d] & {\bitableau{\bi11\amp\bi12\amp\bi12\\\bi22}} \arrow[r] \arrow[d] & {\bitableau{\bi12\amp\bi12\amp\bi12\\\bi22}} \arrow[d] &  &                        &                        &              \\
{\bitableau{\bi11\amp\bi11\amp\bi21\\\bi21}} \arrow[r] \arrow[d] & {\bitableau{\bi11\amp\bi11\amp\bi22\\\bi21}} \arrow[r] \arrow[d] & {\bitableau{\bi11\amp\bi11\amp\bi22\\\bi22}} \arrow[r] \arrow[d] & {\bitableau{\bi11\amp\bi12\amp\bi22\\\bi22}} \arrow[r] \arrow[d] & {\bitableau{\bi12\amp\bi12\amp\bi22\\\bi22}} \arrow[d] &  &                        &                        &              \\
{\bitableau{\bi11\amp\bi21\amp\bi21\\\bi21}} \arrow[r]           & {\bitableau{\bi11\amp\bi21\amp\bi22\\\bi21}} \arrow[r]           & {\bitableau{\bi11\amp\bi22\amp\bi22\\\bi21}} \arrow[r]           & {\bitableau{\bi11\amp\bi22\amp\bi22\\\bi22}} \arrow[r]           & {\bitableau{\bi12\amp\bi22\amp\bi22\\\bi22}}           &  &                        &                        &             
\end{tikzcd}}

\vspace{.1in}

\noindent The undetermined gray pieces are filled with the following components arranged by $a$-weight.

\vspace{.1in}

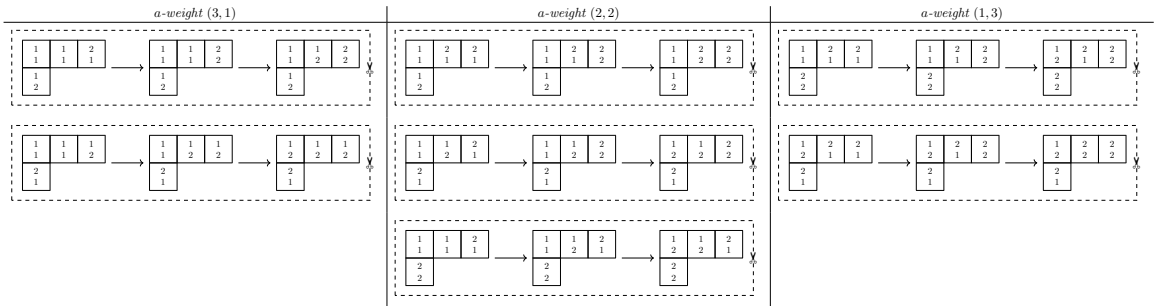
\newcommand{\cutout}[1]{\begin{tikzpicture}
\node[rectangle,minimum width=.5in,ampersand replacement=\&] (m) {\begin{minipage}{3.65in}
#1
\end{minipage}};
\draw[dashed] (m.south west) rectangle (m.north east);
\node[rotate=90] at (1.88in,0in) {\ding{34}};
\path (current bounding box.south west) +(0,-.2) (current bounding box.north east) +(0,.2);
\end{tikzpicture}}

\adjustbox{scale=.5,center}{\begin{tabular}{c|c|c}
$a$-weight $(3,1)$ & $a$-weight $(2,2)$ & $a$-weight $(1,3)$\\ \hline
     \cutout{\begin{tikzcd}
    \bitableau{\bi11\amp\bi11\amp\bi21\\\bi12} \arrow[r] \& \bitableau{\bi11\amp\bi11\amp\bi22\\\bi12} \arrow[r] \& \bitableau{\bi11\amp\bi12\amp\bi22\\\bi12}   
\end{tikzcd}}
 &      \cutout{\begin{tikzcd}
    \bitableau{\bi11\amp\bi21\amp\bi21\\\bi12} \arrow[r] \& \bitableau{\bi11\amp\bi21\amp\bi22\\\bi12} \arrow[r] \& \bitableau{\bi11\amp\bi22\amp\bi22\\\bi12}   
\end{tikzcd}} &      \cutout{\begin{tikzcd}
    \bitableau{\bi11\amp\bi21\amp\bi21\\\bi22} \arrow[r] \& \bitableau{\bi11\amp\bi21\amp\bi22\\\bi22} \arrow[r] \& \bitableau{\bi12\amp\bi21\amp\bi22\\\bi22}   
\end{tikzcd}}\\
\cutout{\begin{tikzcd}
    \bitableau{\bi11\amp\bi11\amp\bi12\\\bi21} \arrow[r] \& \bitableau{\bi11\amp\bi12\amp\bi12\\\bi21} \arrow[r] \& \bitableau{\bi12\amp\bi12\amp\bi12\\\bi21}   
\end{tikzcd}}
 &      \cutout{\begin{tikzcd}
    \bitableau{\bi11\amp\bi12\amp\bi21\\\bi21} \arrow[r] \& \bitableau{\bi11\amp\bi12\amp\bi22\\\bi21} \arrow[r] \& \bitableau{\bi12\amp\bi12\amp\bi22\\\bi21}   
\end{tikzcd}} &      \cutout{\begin{tikzcd}
    \bitableau{\bi12\amp\bi21\amp\bi21\\\bi21} \arrow[r] \& \bitableau{\bi12\amp\bi21\amp\bi22\\\bi21} \arrow[r] \& \bitableau{\bi12\amp\bi22\amp\bi22\\\bi21}   
\end{tikzcd}}\\
 &      \cutout{\begin{tikzcd}
    \bitableau{\bi11\amp\bi11\amp\bi21\\\bi22} \arrow[r] \& \bitableau{\bi11\amp\bi12\amp\bi21\\\bi22} \arrow[r] \& \bitableau{\bi12\amp\bi12\amp\bi21\\\bi22}   
\end{tikzcd}} &    
\end{tabular}}

As before, any way of inserting these segments into the empty gray slots of the same weight will yield a valid $\mathfrak{gl}_2 \times \mathfrak{gl}_2$-crystal.

\end{example}

\subsection{Maximum entry three}

This time we will allow the top entries to go up to $3$ instead of $2$.  These examples tend to get big fast, so instead of showing the full crystal we will just look at those bitableaux where the sort-by-top reading word $w(T)$ is Yamanouchi of a fixed weight -- in this case $(2,1)$.  We note that in general it would be enough to construct a $\mathfrak{gl}_n$-crystal structure on the set of bitableaux of fixed weight with Yamanouchi $w(T)$ -- the rest of the crystal is then determined because the $\mathfrak{gl}_n$ and $\mathfrak{gl}_m$-crystal structures commute. 

For the following example we made some additional assumptions about how this $\mathfrak{gl}_n$-crystal might be determined by a reading word $v(T)$, namely:

\begin{enumerate}
    \item The order in which the top entries are read by $v$ depends only on the bottom entries.
    \item The output of the reading method $v$ on the bitableaux must be the row-reading word of some semistandard Young tableau\footnote{ The first assumption is a property that we hope these crystals will have in general, the second assumption is to limit the number of possible reading orders.}.
\end{enumerate}

\begin{example}\label{ex:sh21} For shape $(2,1)$ and fixed $b$-content $(2,1)$, there are three possible ways for the bottom entries to be distributed. To satisfy the conditions above, when the bottom 2 is in the eastern-most box or southern-most box, the top entries have to be read as follows:

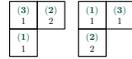
\adjustbox{scale=.5,center}{$\begin{array}{ccc}
     \bitableau{\bi{\mathbf{\textcolor{c4}{(3)}}}1\amp\bi{\mathbf{\textcolor{c4}{(2)}}}2\\\bi{\mathbf{\textcolor{c4}{(1)}}}1} & \bitableau{\bi{\mathbf{\textcolor{c4}{(1)}}}1\amp\bi{\mathbf{\textcolor{c4}{(3)}}}1\\\bi{\mathbf{\textcolor{c4}{(2)}}}2}
\end{array}$}

\vspace{.1in}

When the bottom 2 is in the corner, there are two possible ways for the top entries to be read, with the corner always being read second. We choose the following reading order for illustrative purposes.

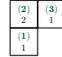
\vspace{.1in}

\adjustbox{scale=.5,center}{$\begin{array}{c}\bitableau{\bi{\mathbf{\textcolor{c4}{(2)}}}2\amp\bi{\mathbf{\textcolor{c4}{(3)}}}1\\\bi{\mathbf{\textcolor{c4}{(1)}}}1}\end{array}$}

\vspace{.1in}

This reading order yields the following crystal (and changing to the other reading order only swaps the tableaux in the center column of the last component).

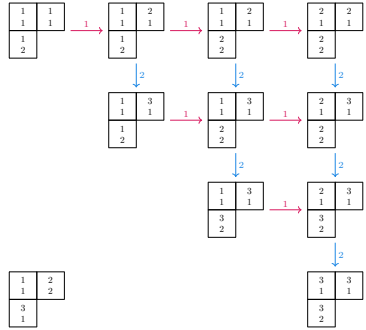
\adjustbox{scale=.5,center}{
\begin{tikzcd}
{\bitableau{\bi11\amp\bi11\\\bi12}} \arrow[r,"1",c1] & {\bitableau{\bi11\amp\bi21\\\bi12}} \arrow[r,"1",c1] \arrow[d,"2",c2] & {\bitableau{\bi11\amp\bi21\\\bi22}} \arrow[r,"1",c1] \arrow[d,"2",c2] & {\bitableau{\bi21\amp\bi21\\\bi22}} \arrow[d,"2",c2] \\
             & {\bitableau{\bi11\amp\bi31\\\bi12}} \arrow[r,"1",c1]           & {\bitableau{\bi11\amp\bi31\\\bi22}} \arrow[r,"1",c1] \arrow[d,"2",c2] & {\bitableau{\bi21\amp\bi31\\\bi22}} \arrow[d,"2",c2] \\
             &                        & {\bitableau{\bi11\amp\bi31\\\bi32}} \arrow[r,"1",c1]           & {\bitableau{\bi21\amp\bi31\\\bi32}} \arrow[d,"2",c2] \\
            \bitableau{\bi11\amp\bi22\\\bi31} &                        &                        & {\bitableau{\bi31\amp\bi31\\\bi32}}          
\end{tikzcd}}

\vspace{.1in}

\adjustbox{scale=.5,center}
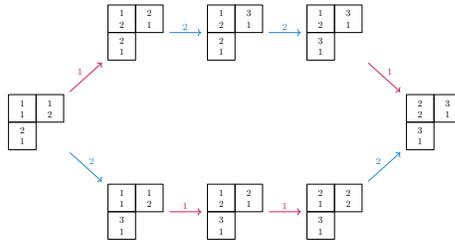
{
\begin{tikzcd}
                         & {\bitableau{\bi12\amp\bi21\\\bi21}} \arrow[r,"2",c2] & {\bitableau{\bi12\amp\bi31\\\bi21}} \arrow[r,"2",c2] & {\bitableau{\bi12\amp\bi31\\\bi31}} \arrow[rd,"1",c1] &    \\
{\bitableau{\bi11\amp\bi12\\\bi21}} \arrow[ru,"1",c1] \arrow[rd,"2",c2] &              &              &               & {\bitableau{\bi22\amp\bi31\\\bi31}} \\
                         & {\bitableau{\bi11\amp\bi12\\\bi31}} \arrow[r,"1",c1] & {\bitableau{\bi12\amp\bi21\\\bi31}} \arrow[r,"1",c1] & {\bitableau{\bi21\amp\bi22\\\bi31}} \arrow[ru,"2",c2] &   
\end{tikzcd}}
So far we are unable to generalize this method of extracting a reading word for bitableaux of arbitrary shape, but we do think it is suggestive of what a general rule might look like.
    
\end{example}

\bibliography{bibliography}

\begin{thebibliography}{AAER17}

\bibitem[AAER17]{adin2017character}
Ron~M. Adin, Christos~A. Athanasiadis, Sergi Elizalde, and Yuval Roichman.
\newblock Character formulas and descents for the hyperoctahedral group.
\newblock {\em Adv. in Appl. Math.}, 87:128--169, 2017.

\bibitem[Bla18]{blasiak2016Kronecker}
Jonah Blasiak.
\newblock Kronecker coefficients for one hook shape.
\newblock {\em S\'em. Lothar. Combin.}, 77:Art. B77c, 40, [2016--2018].

\bibitem[BO07]{ballantine2005combinatorial}
Cristina~M. Ballantine and Rosa~C. Orellana.
\newblock A combinatorial interpretation for the coefficients in the
  {K}ronecker product {$s_{(n-p,p)}\ast s_\lambda$}.
\newblock {\em S\'em. Lothar. Combin.}, 54A:Art. B54Af, 29, 2005/07.

\bibitem[BS17]{bump2017crystal}
D.~Bump and A.~Schilling.
\newblock {\em Crystal Bases: Representations And Combinatorics}.
\newblock World Scientific Publishing Company, 2017.

\bibitem[But94]{butler1994subgroup}
Lynne~M. Butler.
\newblock Subgroup lattices and symmetric functions.
\newblock {\em Mem. Amer. Math. Soc.}, 112(539):vi+160, 1994.

\bibitem[Ful97]{fulton1997young}
William Fulton.
\newblock {\em Young tableaux}, volume~35 of {\em London Mathematical Society
  Student Texts}.
\newblock Cambridge University Press, Cambridge, 1997.
\newblock With applications to representation theory and geometry.

\bibitem[Knu70]{knu:pmg}
Donald~E. Knuth.
\newblock Permutations, matrices, and generalized {Y}oung tableaux.
\newblock {\em Pacific J. Math.}, 34:709--727, 1970.

\bibitem[Liu17]{liu2017simplified}
Ricky~Ini Liu.
\newblock A simplified {K}ronecker rule for one hook shape.
\newblock {\em Proc. Amer. Math. Soc.}, 145(9):3657--3664, 2017.

\bibitem[LS81]{LS:mp}
Alain Lascoux and Marcel-P. Sch\"utzenberger.
\newblock Le mono\"ide plaxique.
\newblock In {\em Noncommutative structures in algebra and geometric
  combinatorics ({N}aples, 1978)}, volume 109 of {\em Quad. ``Ricerca Sci.''},
  pages 129--156. CNR, Rome, 1981.

\end{thebibliography}
\bibliographystyle{alpha}

\end{document}